\documentclass[11pt]{article}
\usepackage{graphicx}
\usepackage{amsmath}
\usepackage{amssymb}
\usepackage{amsthm}
\usepackage{epstopdf}
\usepackage{url}
\usepackage{hyperref}
\usepackage[all]{xy}
\usepackage{todonotes}
\usepackage{hyperref}
\usepackage{caption}
\usepackage{subcaption}
\usepackage{diagrams}

\DeclareMathOperator{\id}{id}

\newtheorem{theorem}{Theorem}

\newtheorem{definition}{Definition}

\newtheorem{lemma}{Lemma}
\newtheorem{proposition}{Proposition}
\newtheorem{prop}{Proposition}

\newtheorem{conjecture}[theorem]{Conjecture}

\def\c{\mathcal{C}}

\def\Z{{\mathbf{Z}}}
\def\Q{\mathbf{Q}}
\def\C{\mathbf{C}}

\def\Aut{\rm Aut}

\def\Q{{\mathbf{Q}}}

\def\c{{\mathcal C}}

\title{Corrections to {\it Uniformity of rational points}  and further comments}
\author{Lucia Caporaso, Joe Harris and Barry Mazur}
\date{\today}                                           

\begin{document}
\maketitle

\section{Introduction}

The purpose of this note is to correct, and enlarge on, an argument in a paper [\ref{1}] published almost a quarter century ago. The question raised in~[\ref{1}] is a simple one to state: given that a curve $C$ of genus $g \geq 2$ defined over a number field $K$ has only finitely many rational points, we ask if the number of points is bounded as $C$ varies. 

In~[\ref{1}] it is asserted that, assuming the truth of the Strong Lang Conjecture (Conjecture~\ref{SLC} below), a very strong form of boundedness holds: for every $g \geq 2$ there is a finite  bound $N(g)$---not depending on $K$!---such that for any number field $K$ there are only finitely many isomorphism classes of curves of genus $g$ defined over $K$ with more than $N(g)$ $K$-rational points. The issue is, in that statement do we mean finitely many isomorphism classes over $K$, or over the algebraic closure $\overline K$? The paper asserts the statement in the stronger form---up to isomorphism over $K$---but the proof establishes only the weaker statement that there are finitely many curves with more than $N(g)$ points up to isomorphism over $\overline K$.

The main purpose of this note is to give a complete argument of the stronger form, which we will do in Sections~\ref{key lemma}  and \ref{main argument}. Of course, if indeed there is a ``universal" bound $N = N(g)$ on the number of points on a curve of genus $g$ defined over an arbitrary number field---with finitely many exceptions for any given $K$---the question of how large $N(g)$ has to be is an intriguing one, and we devote the final chapter to a preliminary discussion of this  and related questions.


We are extremely grateful to Jakob Stix, who pointed out our error, and who helpfully---and generously---entered into detailed discussion about it with us. We are also immensely thankful to Dan Abramovich  for his patient guidance as we wrote this, and for  his more general results  regarding uniform boundedness.

\section{Moduli Spaces}\label{mod} Fix a genus $g >1$.
\begin{itemize} \item ({\bf The coarse moduli space}) Let $M= M_g$, the coarse moduli space of smooth projective  curves of genus $g$; so $M$ is a variety defined over $\Q$.
 
  \item  ({\bf  The rigidified moduli space})   \begin{definition}  A point $p$ in a variety $V$ over  a field $K$  is {\bf rigid in $V$} if there are no nontrivial automorphisms of $V$ (over the algebraic closure  ${\bar K}$)  that fix $p$; i.e., for any automorphism $\alpha:V \to V$ if $\alpha(p) = p$ then $\alpha$ is the identity.\end{definition}  Let ${\mathcal M}_{g,1}$ be the Deligne-Mumford stack  of smooth projective curves $C$ of genus $g$ with one marked point $p \in C$. We will denote by ${\mathcal M}^*$ the open substack of ${\mathcal M}_{g,1}$ corresponding to pairs $(C,p)$ where $C$ is a smooth projective curve  of genus $g$ and $p$ is a rigid point in $C$. (Call such a pair $(C,p)$ a {\bf rigidified curve}.) The stack ${\mathcal M}^*$ has trivial inertia and so ( is a fine moduli space)   representable by an algebraic space   $M^*$)(cf.  92.13 in \cite{12}).  The algebraic space $M^*$ is a quasi-projective scheme (cf.   the classical results of Knudsen \cite{11} and Koll\'ar  \cite{10}).   We note that $M^*$ is a scheme of finite type over $\Q$ and: there is a universal family $\phi: \c_{M^*} \to M^*$, such that for any rigidified curve $(C,p)$ defined over $K$ there is a $K$-point $[(C,p)] \in M^*$ such that the fiber of $\c_{M^*}$ over the point $[(C,p)]$ is isomorphic to $C$.
  
  The forgetful projection $(C,p) \mapsto C$ gives us a mapping $$M^* \longrightarrow M$$ defined over $\Q$ (with one-dimensional fibers).
  
  \begin{prop}\label{rigid} For   $g>1$  there is a finite bound $B_g$ with the property that if  $K$ is a (number) field and $C$ a  smooth projective curve of genus $g$, defined over $K$, such that $|C(K)|  \ > B_g$  there is a $K$-rational rigid point $p$ in $C$.  The curve  $C$ is  (therefore)  represented by a $K$-rational point of $M^*$.\end{prop}

We thank Jakob Stix for communicating a proof  of the fact that one can take $B_g$ to be equal to $82(g-1)$. See  Appendix 2 (section \ref{aut} below); 

\item  ({\bf The moduli space with level structure}) Here it will suffice for us to work over $\C$.  Let $\ell \gg 0$ be a prime and  ${\tilde M}_{g,1}:=M_{g,1}[\ell]$ the moduli space of smooth pointed curves of genus $g$ with full level $\ell$ structure.  That is, $M_{g,1}[\ell]$ classifies pairs $(C,\lambda)$  where $C$ is a smooth pointed curve of genus $g$  (over $\C$) and  (the `level structure') $\lambda$ is an isomorphism of  ${\bf F}_\ell$-vector spaces $$\lambda: {\bf F}_\ell^{2g} \stackrel{\simeq}{\longrightarrow} H_1(C_\C; {\bf F}_\ell).$$  Note that  ${\tilde M}_{g,1}$ is not connected, but this won't bother us. The finite group $G:= {\rm GL}_{2g}({\bf F})$ acts on  ${\tilde M}_{g,1}$ with quotient  $M_{g,1}$. 

  Define  ${\tilde M}^*$ by the following diagram, the upper square being exact:{\footnote{An {\it exact} square $$\xymatrix{A\ar[r]\ar[d] & B\ar[d] \\  D\ar[r] & C}$$  is a commutative square, where the mapping $A \to B\times_CD$  determined by the diagram is an isomorphism.}}

   \begin{equation}\label{exact1} \xymatrix{{\tilde M}^*\ar[r]\ar[d]^G & {\tilde M}_{g,1}\ar[d]^G \\
    M^*\ar[r]\ar[d] &  M_{g,1}\ar[d]\\
    M\ar[r]^= & M.}\end{equation}
    
    \noindent   So the group $G$ acts on ${\tilde M}^*$ with quotient $M^*$ rendering ${\tilde M}^*$ a $G$-torsor over $M^*$ as well.  The fine moduli space   ${\tilde M}^*$ classifies triples $(C,p,\lambda)$ and we  have an exact square of   universal families:
    
      \begin{equation}\label{exact2} \xymatrix{\c_{{\tilde M}^*}\ar[r]^G\ar[d]^{\tilde \phi} & \c_{M^*}\ar[d]^\phi \\
    {\tilde M}^*\ar[r]^G &  M^*.}\end{equation}
    
\noindent These  (i.e., the vertical morphisms) are flat families of smooth projective curves of genus $g$, and  the  group $G$ acts equivariantly, rendering the domains of the horizontal morphisms $G$-{\it torsors} over the corresponding ranges{\footnote{E.g., the mapping $$G\times  {\tilde M}^* \   \longrightarrow\  {\tilde M^*}\times_{M^*}{\tilde M}^*$$ given by $(g,m) \mapsto (m, g(m))$ is an isomorphism.}}.

\item ({\bf General families of rigid  curves})  Let $B$ be a scheme of finite type over $\C$, and  $\phi_B: \c_B \to B$  a flat family of smooth projective {\it rigidified} curves of genus $g$ (over $B$)---that is, such that there is a section $p: B \to \c_B$  having the property that for every point $b$ of $B$ the image point $p(b)$ in the fiber $\c_b$ over $b$ is a  rigid point of that curve $\c_b$.   Since  $M^*$ is the fine moduli space for such objects,  this family comes by pullback from a unique morphism $j:B \to M^*$ and $\phi_B$ fits into a diagram, the upper square being exact:

 \begin{equation}\label{exact3} \xymatrix{\c_B\ar[r] \ar[d]^{\phi_B} & \c_{M^*}\ar[d]^\phi \\
    B\ar[r]^j\ar[dr]^i\ar[d]   &  M^*\ar[d]^k\\
B_0=i(B)\ar[r]^{\hookrightarrow}  \ & M   }\end{equation}
    
Here, by Chevalley's classical theorem, the image of $B$  in $M^*$ (via the  mapping $j$) and in $M$ (via the  mapping $i$)  are constructible sets, so the first is a finite union  of locally closed (irreducible) subvarieties of  $M^*$, and the second  is a finite union  of locally closed (irreducible) subvarieties of  $M$.   We will deal, inductively with all of these subvarieties; but:

\begin{itemize} \item Let  $B'_0$ be any one of the  locally closed (irreducible) subvarieties  in $M$ that is among   components of the constructible set which is the image of $B$  in   $M$, and 
\item let $B'$ be a  locally closed (irreducible) subvariety of  $M^*$ that is \begin{itemize} \item  among  components of the constructible set which is the image of $B$  in   $M^*$ and\item that  contains a Zarkisi-dense open in the  inverse image of $B'_0$ under  $k$.\end{itemize} \end{itemize} 

We   have an analogous diagram as (\ref{exact3})  but\begin{itemize} \item with $B$ replaced with $B'$; and   $B_0$ replaced with $B_0'$; but such that  \item all morphisms are  morphisms of varieties, and  \item where $B_0'$  and $B'$ are locally closed subvarieties of $M$ and $M^*$ respectively.\end{itemize}  

 Removing the primes ($'$) from the terminology we have:

 \begin{equation}\label{Exact4} \xymatrix{\c_B\ar[r]^{\hookrightarrow} \ar[d]^{\phi_B} & \c_{M^*}\ar[d]^\phi \\
    B\ar[r]^{\hookrightarrow}\ar[dr]^i\ar[d]   &  M^*\ar[d]^k\\
B_0=i(B)\ar[r]^{\hookrightarrow}  \ & M   }\end{equation} 
   In diagram (\ref{Exact4}) it is {\it only} the upper square that is exact. These are the diagrams we will be studying. Call such a family of rigid curves, $\c_{B}\to B$, {\bf clean}.  From now on we will assume that our families $\c_B \to B$ are `clean.'

 Augmenting such a {\it clean} family with level structure by tensoring with ${\tilde M}$  (over $M$) with  we might form
 
  \begin{equation}\label{exact4} \xymatrix{\c_B\ar[r]  & B\ar[r]^j  &  M^* &  \c_{M^*}\ar[l] \\ \c_{\tilde B}\ar[r]\ar[u]^{G}\ar[d] & {\tilde B}\ar[r]^{\tilde j}\ar[u]^{G}\ar[d] & {\tilde M}^*\ar[u]^{G}\ar[d]  & \c_{{\tilde M}^*}\ar[d]\ar[l]\ar[u]^{G}\\
\c_{\tilde B_0}\ar[r] &   {\tilde B_0}\ar[r]^{\hookrightarrow}\ar[d]^{G}  & {\tilde M}\ar[d]^{G} & \c_{\tilde M}\ar[l] \\ \ & B_0\ar[r]^{\hookrightarrow}  & M  & \  } \end{equation}

 Here the vertical  mappings in the two exact diagrams
 
   $$\xymatrix{\c_B\ar[d]\ar[r] & \c_{M^*}\ar[d] &\ \ \ \ \  & \c_{\tilde B_0}\ar[r]\ar[d]  &  \c_{\tilde M}\ar[d]\\
   B\ar[r]^{\hookrightarrow} & {M^*} &\ \ \ \ \  & {\tilde B_0}\ar[r]^{\hookrightarrow}  &  \tilde M
   }$$   are flat families of (smooth projective rigidified curves of genus $g$) and---respectively--- flat families of (smooth projective curves of  genus $g$ with level structure).  The arrows labelled ``$G$" are morphisms obtained by passing to the quotient by the natural action of $G$.  All squares  where the vertical arrows are labelled ``$G$" are cartesian and $G$-equivariant.  And note that  the schemes on the bottom line of diagram \ref{exact4}---i.e.,  $B_0 \hookrightarrow M$---do not possess ``universal families."
 
\end{itemize}

\section{A Strengthened Correlation Theorem}\label{key lemma}

\noindent Note: the results of this section are purely geometric, rather than arithmetic; objects will be varieties defined over $\C$. Moreover, we will be dealing entirely with birational properties, so we will feel free to restrict to open subsets where convenient. Thus, for example, when we say that the fibers of a morphism are curves of genus $g$, we will mean that they are open subsets of a curve whose normalization is a smooth projective curve of genus $g$.

For our purposes, we will need the following slightly  strengthened version of the {\it Correlation Theorem}, the key geometric lemma   (i.e., Proposition 3.1) of [2]:

\begin{proposition}\label{basic lemma}  With the notation of the previous section, if the map $B \stackrel{j}{\longrightarrow} M^*$ is generically finite, then for $n \gg 0$ the fiber power $\c^n_B$   (over $B$)  is of general type.
\end{proposition}

{\bf Remarks:}
\begin{enumerate} \item This is stronger than the Correlation Theorem in just one respect: we are only assuming  that the map $j : B \to M^* $ is generically finite, not that the projection $ B \to B_0 \hookrightarrow M$ is generically finite:
$$\xymatrix{B\ar[r]^j\ar[dr]^{j_0}\ar[d]^h &  M^*\ar[d]\\
B{_0}\ar[r]^{\hookrightarrow} & M}$$
\item\label{2}   There is an  obvious bifurcation: either the map $j_0 : B \to M$ is generically finite, or it has generically one-dimensional fibers. In the former case, Proposition (3.1) of [CHM] applies, and we're done; thus we can, and will, assume that the general fiber of $j_0$ has dimension 1, and more specifically that:

\begin{equation}\label{posdimfib} B \subset M^* {\rm \ is \ the\  inverse\  image\  of\  } B_0 {\rm \ in\  } M.\end{equation} 

\begin{lemma}\label{posdimfib2} Under hypothesis \ref{posdimfib} above, the morphism
\begin{equation}\label{tildes}{\tilde B} \to {\tilde B}_0\end{equation}
is a smooth morphism  with fibers that are curves   of genus $g$.\end{lemma}

{\bf Proof:}  First,  the morphism  ${\tilde M}^*  \to {\tilde M}$  has the property that its fibers are curves  (whose smooth projective completions are) of genus $g$.   This is  because  ${\tilde M}$  is a fine moduli space, and the operation of ``tilde"  ($\tilde \ $) and ``star" ($*$) commute, so that the fiber of a point $[(C,\lambda)]$ in  ${\tilde M}$ is given by $([(C,\lambda)],p)$  where $p$ ranges through the locus of all rigid points of $C$.

Also, by  \ref{posdimfib},  we also have that:
\begin{equation}\label{posdimfib3}{\tilde B} \subset {\tilde M}^* {\rm \ is \ the\  inverse\  image\  of\  } {\tilde B}_0 {\rm \ in\  } {\tilde M}\end{equation}
so that 
$$\xymatrix{{\tilde B}\ar[r]^{\tilde j}\ar[d]^{\tilde h} &  {\tilde  M}^*\ar[d]\\
{\tilde B{_0}}\ar[r]^{\hookrightarrow} &{\tilde  M}.}$$

is an exact square, and therefore the fibers of ${\tilde B} \to {\tilde B}_0$ are pullbacks of the fibers of ${\tilde M}^*  \to {\tilde M}$.


\item However if it were true  (but  it is not true, generally)   that  $h:B \to B_0$  has fibers that are  curves of genus $g$  we would then be done: a high fiber power $\c_{B_0}^n$  (over $B_0$) would be of general type by the correlation theorem, and the projection  $$\c^n_B: =\c^n_{B_0}\times_{B_0}B\  \stackrel{1\times h}{\longrightarrow}\  \c^n_{B_0}\times_{B_0}{B_0}= \  \c^n_{B_0}$$ would have fibers  that generically  are curves of genus $g$. 
so---by \cite{9}--- it would be of general type as well. Another way of thinking about  the obstruction to proving Proposition~\ref{basic lemma} is that there may not exist a tautological family over $B_0$.\end{enumerate}

   To prove Proposition  \ref{basic lemma} we use a proposition supplied by Kenneth Ascher and Amos Turchet. 
 Consider the diagonal action of $G$ on fiber powers $\c^n_{\tilde B}$ and  $\c^n_{\tilde B_0}$  (these powers being taken over ${\tilde B}$ and ${\tilde B}_0$ respectively){\footnote{See Subsection \ref{fp} below. The action of $g\in G$ is induced, in the evident way, from the action on isomorphism classes $(C,\lambda) \mapsto  (C,\lambda\cdot g)$.}}.


\begin{proposition}\label{Dan lemma}
Keeping to the notation and hypotheses of section \ref{mod}, for $n$ sufficiently large the quotient  $\c^n_{{\tilde B}_0}/G$ of $\c^n_{{\tilde B}_0}$ (under the diagonal action of $G$) is of general type.
\end{proposition}

\begin{proof}
This is just Theorem 1.7  in \cite{8}, in the special case $D=0$. (The hypotheses in \cite{8} require that the base $B$ be smooth and projective, but we can always achieve this by completing the family, applying stable reduction and resolving the singularities of the new base. If a base change is required in the process of stable reduction, the group of the cover can be incorporated in $G$.)
It should be noted that a major part of the work in \cite{8} is to extend the original theorem  to the setting of log varieties, which does not concern us; what is new and useful for us is the incorporation of the group $G$.
\end{proof}
\subsection{Fiber Powers}\label{fp}
The group  $G$ acts equivariantly on the objects in the exact diagram

\begin{equation}\label{exact7}\xymatrix{ \c_{\tilde B}\ar[r]\ar[d] & \c_{{\tilde B}_0}\ar[d] \\
{\tilde B}\ar[r]  &{\tilde B}_0}.\end{equation}   
\noindent   The square (\ref{exact7}) is exact since the  $\c$'s involved  are the universal families of curves over ${\tilde B}\to {\tilde B}_0$ (that is, pullbacks of the universal family over the fine moduli space ${\tilde M}_g$).  For any $n\ge 1$ let    $$\c^n_{\tilde B}:= \c_{\tilde B}\times_{\tilde B}\c_{\tilde B}\times_{\tilde B}\dots \times_{\tilde B}\c_{\tilde B},$$  i.e., the $n$-fold  power of $ \c_{\tilde B}$ over ${\tilde B}$, with the group $G$ acting on $\c^n_{\tilde B}$ by the  diagonal action. This action is equivariant for the natural projection $\c^n_{\tilde B}\to {\tilde B}$.   The 
map \begin{equation}\label{Gtor}
\c^n_{\tilde B} \to \c^n_{\tilde B_0}\end{equation}  is a morphism of $G$-torsors.

\begin{lemma}\label{fp1}
 For $n \ge 1$ the natural map   $\c^n_{\tilde B} \to \c^n_B$ identifies  $\c^n_B$  (the corresponding fiber power $\c^n_B$ of our original family $\c \to B$)  with $\c^n_{\tilde B}/G$,  the quotient of  $\c^n_{\tilde B}$ by the action of $G$. \end{lemma}

{\bf Proof:}  The  natural map referred to arises from the the following natural map, valid for any three schemes over a scheme $S$, call them 
$$\xymatrix{ X\ar[rd]   &  {\tilde S}\ar[d] & Y\ar[ld]\\
\ & S & \ }.$$

Put: ${\tilde X}:=X\times_S{\tilde S}$ and  ${\tilde Y}:=Y\times_S{\tilde S}.$
We have   canonical isomorphisms of ${\tilde S}$-schemes:  $$X\times_SY\times_S{\tilde S}\simeq  (X\times_S{\tilde S})\times_{\tilde S}(Y\times_S{\tilde S})\simeq {\tilde X}\times_{\tilde S}{\tilde Y},$$

E.g., on points $x, {\tilde s}, y$ of $X, {\tilde S}, Y$ all of which map to the same  point  $s$ of $S$,  it's given by $$x \times y \times {\tilde s}  \mapsto  (x\times {\tilde s})\times (y \times {\tilde s}).$$

 Proceeding inductively on $n$ this gives us a canonical isomorphism
   \begin{equation}\label{nfold}\c_B^n\times_B{\tilde B}:=\ \ \ \c_B\times_B\c_B\times_B\dots \c_B\times_B{\tilde B}\ \ \    \stackrel{\simeq}{\longrightarrow} \ \ \ \c^n_{\tilde B}:= \ \ \  \c_{\tilde B}\times_{\tilde B}\c_B\times_{\tilde B}\dots \c_B\times_B{\tilde B}, \end{equation}
   
   \noindent by taking $S:=B$, ${\tilde S}:={\tilde B}$, $X :=  \c_B$,  $Y:= \c_B^{n-1}$.  Equation (\ref{nfold}) is an equivariant isomorphism for the action of the group $G$,  which acts diagonally on the right hand side  and as for the left hand side,  an element $g \in G$ acts on the fiber product $\c_B^n\times_B{\tilde B}$ by the identity on the first factor  (and as it has  been defined to act, on the second). 
   The map  ${\tilde B}\to B={\tilde B}/G$    (i.e., the map  that exhibits $B$ as the quotient of ${\tilde B}$ under the action of $G$) induces a mapping $\c_B^n\times_B{\tilde B}\to  \c_B^n\times_BB=  \c_B^n$.
   
     Since the quotient of  ${\tilde B}$ under the action of $G$ is $B$, the quotient of   $\c_B^n\times_B{\tilde B}$  under the action of $G$ is $B$ is canonically isomorphic to  $\c_B^n$, and we have the commutative diagram:
  $$ \xymatrix{ \c_B^n\times_B{\tilde B}\ar[r]^{\simeq}\ar[d] &  \c^n_{\tilde B}\ar[d]\\
 \c_B^n\ar[r]^{\simeq} & \c^n_{\tilde B}/G}.$$
.

 We also have the following lemma:
\begin{lemma}\label{reallem1}
   For $n \ge 1$
the fibers of the  map of quotients by the action of $G$ 
\begin{equation}
\c^n_{\tilde B}/G \to \c^n_{\tilde B_0}/G\end{equation}
are {\it generically} curves of genus $g$.
 \end{lemma}

   The proof of the above lemma is given in Appendix 1  (Section  \ref{app1}) below.
   
   \vskip10pt 
{\bf Proof of Proposition \ref{basic lemma}: }  By Proposition \ref{Dan lemma} we have that   for $n \gg 0$  $\c^n_{{\tilde B}_0}/G$ is of general type. By Lemmas \ref{fp1} and  the above lemma,  the mapping  $\c^n_B  \to \c^n_{\tilde B_0}/G$ \   has fibers that are curves of genus $\geq 2$; i.e., that are of general type. By \cite{9}, it follows that   $\c^n_B$ is of general type.

\section{The boundedness argument, given  Lemma~\ref{basic lemma}}\label{main argument}

Let us first state the version of the Lang conjecture we will be invoking.

\begin{conjecture}[Strong Lang]\label{SLC}
Let $X$ be a variety of general type, defined over a number field $K$. There is then a proper subvariety $Z \subset X$ such that for any finite extension $L$ of $K$, $\#(X \setminus Z)(L) < \infty$; that is, all but finitely many $L$-rational points of $X$ lie in $Z$.
\end{conjecture}

Given this and Proposition~\ref{basic lemma}  of section 3, we can deduce the 

\begin{theorem}\label{main theorem}
Assume Conjecture~\ref{SLC} above. If $\pi : \c \to B$ is a family of pointed curves without automorphisms, defined over $\Q$, such that the induced map $\phi : B \to M^*$ is finite, then there is then an integer $N$ such that for any number field $K$,
$$
\# \{ b \in B(K) \mid \#C_b(K) > N\} < \infty
$$
\end{theorem}

\begin{proof}
We will prove an a priori weaker form of this: we will show that there exists a nonempty open subset $U \subset B$ and an integer $N$ such that for any number field $K$,
$$
\# \{ b \in U(K) \mid \#C_b(K) > N\} < \infty;
$$
Theorem~\ref{main theorem} will then follow by Noetherian induction.

To prove this, let $\pi_n : \c^n_B \to B$ be the $n$th fiber power of the family $\c \to B$. By Lemma~\ref{basic lemma}, for large $n$ the fiber power $\c^n_B$ will be of general type. By the Strong Lang Conjecture, then, there will be a proper subvariety $Z \subset \c^n_B$ such that for any number field $K$, all but finitely many $K$-rational points of $\c^n_B$ lie in $Z$; that is,
$$
\#(\c^n_B \setminus Z)(K) < \infty.
$$ 

We now define a sequence of subvarieties $Z_k \subset \c^k_B$ inductively as follows. We start with $Z_n = Z$, and let
$$
Z_{n-1} = \{ b \in \c^{n-1}_B \mid \pi_{n, n-1}^{-1}(b) \subset Z_n \}
$$
where $\pi_{n,n-1} : \c^n_B \to \c^{n-1}_B$ is the projection; similarly, given $Z_k$ we set
$$
Z_{k-1} = \{ b \in \c^{k-1}_B \mid \pi_{k,k-1}^{-1}(b) \subset Z_k \}
$$
where $\pi_{k,k-1} : \c^k_B \to \c^{k-1}_B$ is the projection. We arrive at a tower of spaces and closed subvarieties:

\small{\begin{diagram}
& & Z = Z_n & \subset & \c^n_B \\
& & & & \dTo_{\pi_{n,n-1}} \\
& & Z_{n-1} & \subset &\c^{n-1}_B \\
& & & & \dTo_{\pi_{n-1, n-2}} \\
& & & & \vdots \\
& & & & \dTo_{\pi_{2,1}} \\
& & Z_1 & \subset & \c \\
& & & & \dTo_{\pi = \pi_{1,0}} \\
& & Z_{0} & \subset & B \\
\end{diagram}}

\noindent where the $k$-th story in this tower has the structure:

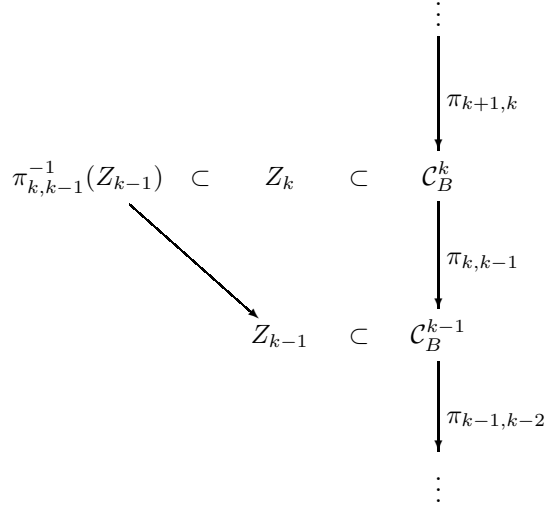
\small{\begin{diagram}
& & & & \vdots \\
& & & & \dTo_{\pi_{k+1,k}} \\
\pi^{-1}_{k,k-1}(Z_{k-1})\quad & \subset & Z_{k} & \subset &\c^{k}_B \\
& \rdTo & & & \dTo_{\pi_{k,k-1}} \\
& & Z_{k-1} & \subset &\c^{k-1}_B \\
& & & & \dTo_{\pi_{k-1, k-2}} \\
& & & & \vdots \\
\end{diagram}}

Note that since $Z \subsetneq \c^n_B$ and $\pi_n^{-1}(Z_0) \subset Z$, we necessarily have $Z_0 \subsetneq B$.

Now fix for the moment a value of $k$ with $1 \leq k \leq n$. Every irreducible component $W_\alpha \subset Z_k$  will  either be the preimage of a subvariety in $\c^{k-1}_B$, or will map onto its image in $\c^{k-1}_B$ with degree $d_\alpha$. Let $d_k$ be the sum of the degrees $d_\alpha$, so that for any $p \in \c^{k-1}_B$, either $\#(\pi_{k,k-1}^{-1}(p)\cap Z_k) \leq d_k$, or $\pi_{k,k-1}^{-1}(p) \subset Z_k$.

Finally, let $N$ be the maximum of the $d_k$, and set $U = B \setminus Z_0$. We claim that for any number field $K$,
$$
\# \{ b \in U(K) \mid \#C_b(K) > N\} < \infty;
$$
as noted above, Theorem~\ref{main theorem} will follow by Noetherian induction. To see this, restrict our family and all fiber powers to the open subset $U \subset B$; similarly, replace $Z$ by its intersection with $\pi_n^{-1}(U)$. Fix a number field $K$, and let 
$$
\Sigma = \{ (\c^n_U \setminus Z)(K) \},
$$
and let $\Sigma_0 = \pi_n(\Sigma)$ be its image; by hypothesis, this is a finite subset of $U$.

We claim finally that \emph{for any $b \in B(K) \setminus \Sigma_0$, we have $\#(X_b(K)) \leq N$}. To see this, let $b \in B(K)$ be any $K$-rational point, and suppose that $\#(X_b(K)) > N$.
Since $b \notin \Sigma_0$, all $K$-rational points of $\c^n_B$ lying over $b$ must lie in $Z$. Pick any $n-1$ points $p_1,\dots,p_{n-1} \in X_b(K)$, and consider the points $\{ (X_b, p_1,\dots,p_{n-1}, p) \mid p \in X_b(K) \} \subset \pi_{n,n-1}^{-1}((X_b, p_1,\dots,p_{n-1}))$ in the fiber of $\c^n_B$ over $(X_b, p_1,\dots,p_{n-1}) \in \c^{n-1}_B$. Since there are by hypothesis more than $N \geq d_n$ such points, we conclude that \emph{$Z=Z_n$ must contain the fiber of $\c^n_B$ over $(X_b, p_1,\dots,p_{n-1}) \in \c^{n-1}_B$}; in other words, $(X_b, p_1,\dots,p_{n-1}) \in Z_{n-1}$.

The same argument applies sequentially to show that $(X_b, p_1,\dots,p_{n-2}) \in Z_{n-2}$, and so on; ultimately, we deduce that $b \in Z_0$, establishing our claim.
\end{proof}
\section{Behavior of $N(g)$ as $g$ tends to $\infty$} 

For $C$  a smooth projective, irreducible curve of genus $g > 1$ defined over a number field $K$ let $Aut_K(C)$ be the group of automorphisms of $C$ defined over $K$.  The group  $Aut_K(C)$ acts naturally on the  set $C(K)$ of $K$-rational points of $C$. Let $\nu(C;K)$ denote the number of  $Aut_K(C)$-orbits in $C(K)$ under that natural action.   So, of course,   $\nu(C;K) \le |C(K)|$ and therefore any uniform upper bound established for $|C(K)|$  is valid for  $\nu(C;K)$ as well. 

Define $\nu(g)$ to be the smallest integer that has the property that for each number field $K$ there are only finitely many curves $C$ of genus $g$  defined over $K$  with  the property that  $\nu(C;K)$ is strictly greater than $\nu(g)$.  By what we have shown,  assuming SLC, $\nu(g)$ is finite  for every $g >1$.

 If one feels that there is a fair chance for Conjecture 1 to be true,  and hence for $\nu(g)$ to be finite, one might wonder about the asymptotic behavior of $\nu(g)$ as $g$ tends to infinity. Needless to say, we have no  real evidence to make any conjectures, or precise predictions, but we set:

$$\nu_*:= \liminf_{g\to \infty} \nu(g)/g$$
{and}
$$\nu^*:= \limsup_{g\to \infty} \nu(g)/g.$$

Note that curves in ${\bf P}^1\times {\bf P}^1$  of bidegree $(2,g+1)$ are of arithmetic genus $g$, and form a linear system of dimension $3(g+2) - 1$. Given  $3(g+2) - 1$ general points $p_1,\dots,p_{3g+5} \in {\bf P}^1\times {\bf P}^1({\bf Q})$, accordingly, there will be  a smooth curve $C$ defined over ${\bf Q}$ and passing through them. Moreover, since $C$ is a general hyperelliptic curve, its automorphism group is equal to $\Z/2$, consisting of the identity and the hyperelliptic involution; and since no two of the points $p_i$ lie in the same fiber of ${\bf P}^1\times {\bf P}^1$ over ${\bf P}^1$, no two are conjugate under the automorphism group of $C$. Thus we have $\nu(C,\Q) \geq 3g+5$ and hence $\nu(g) \geq 3g+5$.

 We have accordingly:

\begin{equation}\label{N} 3 \le \nu_* \le \nu^*.\end{equation}

  Some natural questions:
 
  \begin{enumerate}\item  Is $\nu^*$, or perhaps only $\nu_*$, or neither of them, finite?
  
  \item Are both inequalities in Equation \ref{N}  equalities? (or is one of them, or neither)?
  
  \item Let $M_{g,n}^*$ denote the moduli space of projective smooth curves of genus $g$ with $n$ {\it distinct} marked rigid points.  For $K$ a number field let $d_{g,n}(K)$ denote the dimension of the Zariski-closure in $M_{g,n}^*$ of the set of $K$-rational points  $M_{g,n}^*(K)$. Now define $d_{g,n}:= \max_Kd_{g,n}(K)$ where the maximum is taken over all number fields $K$.  The discussion in this note shows that the conjecture SLC implies that---for fixed  $g \ge 2$---if $n\gg_g0$,  then  $d_{g,n}=0$. What else can one say---or even just conjecture---about these dimensions?    For example, might $d_{g,n}$  be decreasing (albeit not necessarily strictly) for fixed $g$ and increasing $n$?
   \end{enumerate}

\section{Appendix 1: Proof of Lemma \ref{reallem1} }\label{app1}
Recall: 
\begin{quote}\begin{lemma}\label{reallem10}
  For $n \ge 1$
the fibers of the  map of quotients by the action of $G$ 
\begin{equation}\label{Gquot}
\c^n_{\tilde B}/G \to \c^n_{\tilde B_0}/G\end{equation}
are {\it generically}  curves  of genus $g$.
 \end{lemma}
\end{quote} 

The statement of Lemma  \ref{reallem10} being geometric, we work over $\C$; and since we are only interested in fibers, we may assume that  $B_0$ is a point.   This point $B_0$   (in ${\mathcal M}_g$)  classifies  a single isomorphism class of curves (of genus $g>1$); call one curve in that isomorphism class  $C$.  If we want to refer to  that isomorphism class  as a whole, we'll denote it $[C]$.
   \subsection{ What is ${\tilde B}_0$?} Consider now  ${\tilde B}_0$ which  classifies isomorphism classes of pairs $(C,\lambda)$  where $C$ is a curve in the isomorphism class $[C]$ equipped with a level structure $\lambda$  on it. We have chosen our level structure so that such pairs are rigid: $C$ has no nontrivial automorphisms that preserve that level structure $\lambda$.  Let $G$ be, as we had before, the group of automorphisms of the level structure.  

 More specifically,  for any curve $X$ (of our fixed genus $g >1$) we have specified an  $\ell$ such that  no automorphism of a curve of genus $g$ leaves fixed a basis of $H_1(X, Z/\ell Z)\simeq  (Z/\ell Z)^{2g}$.  By definition a {\it level structure}  on $X$ is a specific isomorphism   $H_1(X, Z/\ell Z)\stackrel{\lambda}{\longrightarrow}   (Z/\ell Z)^{2g}$; and $G= GL_{2g}(Z/\ell Z)$ acts naturally on level structures (by right-composition:  $\lambda\mapsto \lambda\cdot g^{-1}$); hence---since $B_0$ is just one point---$G$ acts transitively on the set  ${\tilde B}_0$.  

 Consider $\Gamma:=$ the full automorphism group of the curve $C$ (the curve classified by the point $B_0$).  Any automorphism of a curve $X$ induces an automorphism of  $H_1(X, \Z/\ell\Z)$ and so induces a permutation of level structures on $X$. Fixing such a curve $X=C$ we get a homomorphism  $\Gamma \to G$; it is injective since the curve $C$  with a level structure is rigid.  In other words---given our fixed curve $C$--- the image of $\Gamma$ in $G$  is the isotropy subgroup of $G$ relative to its (transitive) action on the finite set  ${\tilde B}_0$.  Consequently,

 \begin{lemma} Making a choice of curve and  level structure $(C, \lambda)$ there is a natural identification;
\begin{equation} {\tilde B}_0 \stackrel{\simeq}{\longrightarrow} G/\Gamma.\end{equation}\end{lemma} 

 \subsection{ What is $\c_{{\tilde B}_0}$?}
 Now let's pass to considering $\c_{{\tilde B}_0}$; i.e.,  the union of the actual curves in the isomorphism class  ``$[C]$"   with their level structures $\lambda$  (that are classified by the corresponding points $(C, \lambda)$  in the finite set ${\tilde B_0}$). A point in  $\c_{{\tilde B}_0}$ is a triple $(C, p; \lambda)$   where $C$ is---as will always be, in this discussion---`classified by' the point $B_0$, $p \in C$ and $$ (Z/\ell Z)^{2g}\stackrel{\lambda}{\longrightarrow}  H_1(C, Z/\ell Z)$$ is a level structure. There is a natural action of $G$  on $\c_{{\tilde B}_0}$.  That is: \begin{equation}\label{Gact} g(C, p;  \lambda):= (C, g(p); \lambda\cdot g^{-1}).\end{equation}

giving us $G$-equivariant mappings
\begin{equation}\label{tB}  \c_{{\tilde B}_0}   \stackrel{\pi}{\longrightarrow}   {\tilde B}_0  \simeq G/\Gamma\end{equation}
every fiber of which is a curve of genus $g$---these being just our  curves ``$C$"  with different level structures.

 \subsection{What is the quotient of $\c_{{\tilde B}_0}$ by the action of $G$?}

\begin{lemma}  Fix a curve  and level structure $(C,\lambda)$ classified by a point in ${\tilde B}_0$.  After passing to the quotient by $G$ the  ($G$-equivariant) mapping \ref{tB}  induces:
\begin{equation}\label{tB2}  \c_{{\tilde B}_0}/G   \stackrel{\pi}{\longrightarrow}   {\tilde B}_0/G  = B_0\end{equation}
 \noindent the fibers being curves isomorphic to the quotient curve $C/\Gamma$. \end{lemma} 
 
 {\bf Proof:}  This follows from the fact that the image of $\Gamma$ in $G$  is the isotropy subgroup of $G$ relative to its (transitive) action on ${\tilde B}_0$.
 
 \subsection{  What is $B$? } \ $B$ consists of isomorphism classes of pairs $(C,q)$  where $C$ is a curve classified by the point $B_0$  and $q $ is a {\it rigid } point on $C$.
  
  \begin{lemma}\label{diag} Fixing a curve $C$  with moduli point $B_0 \in M_g$,  let $C^*$ denote the Zariski open subset of rigid points in $C$.  We have an isomorphism
  $$B    \stackrel{\simeq}{\longrightarrow}  C^*/\Gamma.$$\end{lemma} 

 {\bf Proof:}   This is evident, but one might also notice that $C^*$ is a $\Gamma$-torsor over $B$, as follows from the definition of rigidity.

 \subsection{ What is  ${\tilde B}$?} 
  The cover    ${\tilde B}$ of $B$  consists of isomorphism classes of triples $(C,q;\lambda)$ with $C$ having moduli point $[C] = B_0$,  $q$ a rigid point on $C$ and $\lambda$ a level structure on $C$. Now just consider the pair $(C,\lambda)$.  This pair has no nontrivial automorphisms, so as $q$ ranges through the (rigid) points of  $C$, we get that      
  \begin{lemma}  {\it Fixing a curve $C$  with moduli point $B_0$} ,
 \begin{enumerate} \item The ($G$-equivariant) mapping  \begin{equation}\label{tB0} {\tilde B} \stackrel{\psi}{\longrightarrow}  {\tilde B}_0 =G/\Gamma\end{equation}  is surjective with fibers isomorphic to $C^*$.
 \item  The quotient of (\ref{tB0}) by the action of $G$ induces a mapping    \begin{equation}\label{tB1} {\tilde B}/G \stackrel{{\bar \psi}}{\longrightarrow}  {\tilde B}_0/G=B_0 \end{equation} with fibers isomorphic to $C^*/\Gamma$.\end{enumerate}    \end{lemma}

 \subsection{ What is  $\c_{\tilde B}$?} 
   Consider  the mapping
 \begin{equation}\label{ct} \c_{\tilde B} \to {\tilde B}.\end{equation}
  
 A  point ${\tilde c}$ of $\c_{\tilde B}$ is given by an isomorphism class of $4$-tuples   $(C,q;\lambda; p)$ where  $(C,q;\lambda)$ comprises the coordinates of the point of  ${\tilde B}$ over which  ${\tilde c}$ lies, and $p\in C$ is a point of $C$. So (\ref{ct}) is a family of curves whose fibers are all isomorphic to $C$  (over the base which is isomorphic to $C^*$).
 \begin{lemma}\label{diag2}    We have an exact commutative `$G$-equivariant'  diagram
$$\xymatrix{\c_{\tilde B}\ar[r]\ar[d] &  \c_{{\tilde B}_0}\ar[d] \\
  {\tilde B}\ar[r] & {\tilde B}_0= G/\Gamma}$$ where the fibers of the vertical maps are  isomorphic to $C$ and the  fibers of the horizontal maps are  isomorphic to $C^*$.
  
 \end{lemma}
  {\bf Proof:}     The vertical map sends the point ${\tilde c}\in \c_{\tilde B}$ represented by the $4$-tuple  $(C,q;\lambda; p)$ to  the point in ${\tilde B}$ represented by the triple  $(C,q;\lambda)$ while the horizontal map sends it to  $(C, \lambda; p)$.  In either case the `retention' of  a level structure $\lambda$ (under either of these `forgetful mappings')---guaranteeing  the fact that  $(C, \lambda)$ admits no nontrivial automorphisms---tells us that the fibers of these projections are as claimed in the lemma.

\subsection{Specializing Lemma \ref{diag2} to a point  ${\tilde b}_0  \in  {\tilde B}_0 $}
Consider, now, the pullback of  the above commutative square to  a point ${\tilde b}_0  \in  {\tilde B}_0  = G/\Gamma$.    Let ${\mathcal F} \subset  \c_{\tilde B}$ denote the fiber over ${\tilde b}_0 \in  {\tilde B}_0$ of the mapping   $$\c_{\tilde B}\to   {\tilde B}_0= G/\Gamma,$$  so that the pullback of the  diagram in Lemma \ref{diag2}   to the point  ${\tilde b}_0  \in  {\tilde B}_0$ yields an exact commutative `$\Gamma$-equivariant'  diagram   \begin{equation}\label{F}\xymatrix{{\mathcal F} \ar[r]\ar[d] &  C \cong \c_{{\tilde b}_0}\ar[d] \\
C^*\cong  {\tilde B}_{{\tilde b}_0}\ar[r] & {\tilde b}_0}.\end{equation}     This diagram may be  written simply as a `$\Gamma$-equivariant' isomorphism \begin{equation}\label{equivisom}{\mathcal F} \cong C\times C^*\end{equation}  where we note that the restriction of  the action of $G$ (on $ \c_{\tilde B}$) to $\Gamma \subset G$  stabilizes ${\mathcal F}$, and the  action of $\Gamma$ on the range $C\times C^*$ is the natural diagonal action; i.e. $$\gamma(x,y) = (\gamma(x), \gamma(y)).$$

 We propose to show that the fibers of the mapping  \begin{equation}\label{newfiber}{\mathcal F}/\Gamma   \longrightarrow C/\Gamma\end{equation}  (in the quotient by the action of $\Gamma$ on the top horizontal morphism of the above diagram (\ref{F})  are (generically) curves in the isomorphism class $[C]$.   More specifically, this is true for the fibers of (\ref{newfiber}) over points in the Zariski dense open $C^*/\Gamma  \subset   C/\Gamma.$   We focus, then, on $$(C^*\times C^*)/\Gamma  \subset (C^*\times C)/\Gamma \cong {\mathcal F}.$$
\begin{lemma}
Consider the  projection  \begin{equation}\label{Gammaquot} (C^*\times C^*)/\Gamma  \to C^*/\Gamma.\end{equation}   Fixing any point $x \in C^*$, the mapping $$C^* \stackrel{\alpha}{\longrightarrow} (C^*\times C^*)/\Gamma$$ given by $$y \mapsto  \ {\rm the\ image\ of \ } (x,y)\ {\rm in\ } (C^*\times C^*)/\Gamma$$\vskip10pt  identifies $C^*$ with  the fiber   of (\ref{Gammaquot})  over the image of $x$ in  $ C/\Gamma$. \end{lemma} 
  {\bf Proof:}  That $\alpha$ maps $C^*$ surjectively onto that fiber is clear: if $(x',y') \in C^*\times C^*$ maps to a point $z$  in that fiber, we can find a $\gamma \in \Gamma$ such that $\gamma(x')=x$.  Taking  $y :=\gamma(y')$ we have that the image of $y$ is $z$. But   $\alpha$ is also injective,  since if for $y, y' \in C^*$  there were an element $\gamma \in \Gamma$ such that $\gamma(x,y) = \gamma(x,y')$ we would have $\gamma(x) = x$, which would contradict the rigidity of the point $x\in C^*$. 
\subsection{Returning to Lemma \ref{diag2}}  We are now ready to consider the quotient of the diagram in Lemma \ref{diag2} by the (equivariant) action of the group $G$.
  
  We get the  commutative ({\it but not necessarily exact}) diagram:
  
 \begin{equation}\label{F/G}\xymatrix{ \c_B\ar[d] & \c_{\tilde B}/G\ar[l]^\simeq\ar[r]^f\ar[d] &  \c_{{\tilde B}_0}/G\ar[d]^{\bar \pi} \\
  B & {\tilde B}/G\ar[l]^\simeq\ar[r]^{\bar\psi} & {\tilde B}_0/G= B_0},\end{equation}
  where  ${\bar \psi}$ has fibers isomorphic to $C^*/\Gamma$  and  ${\bar \pi}$  has fibers isomorphic to $C/\Gamma$.  The two  unlabeled vertical morphisms have fibers isomorphic to the curve $C$.

  Returning to the notation of  diagram (\ref{F/G}) we have:
  
  \begin{proposition} The fibers of the mapping  $$\c_{\tilde B}/G \stackrel{f}{\longrightarrow}   \c_{{\tilde B}_0}/G$$ are  (generically)  curves of genus $g$.\end{proposition} 
  Let $n \ge 1$.  Let $$\c_{\tilde B}^n = \c_{\tilde B}\times_{\tilde B}\c_{\tilde B}\times_{\tilde B}\dots \times_{\tilde B}\c_{\tilde B}\ \ \ \ {\rm i.e.\ } n \ {\rm times},$$
  \noindent as in subsection \ref{fp} above; and ditto for  $\c_{{\tilde B}_0}^n$.   
  
  We let the group $G$ act diagonally{\footnote{as in  Subsection \ref{fp} and as in Equation \ref{Gact} above}}.   It was only for notational convenience that we worked, above,  with the case $n=1$. The same arguments, word for word, allow us (for general $n
  \ge 1$) to get, after passing to quotients by $G$:

  \begin{proposition}  The fibers of the mapping  $$\c_{\tilde B}^n/G \to  \c_{{\tilde B}_0}^n/G$$ are generically curves of genus $g$.\end{proposition}

\section{{\bf Appendix 2:}   Automorphisms of curves: a lemma of Jakob Stix}\label{aut} 
\begin{prop}\label{bound}
Let $C$ be a smooth projective curve of genus $>1$, and let $\Sigma \subset C$ be the set of points of $C$ fixed by some automorphism of $C$ other than the identity. Then $|\Sigma|$ admits some finite upper bound  $B_g < \infty$, dependent only on the genus $g >1$.
\end{prop}
\vskip10pt

{\bf Remark:} Although we only need to know that there is some finite upper bound  $B_g < \infty$ for the purposes of application to Proposition \ref{rigid} in Section \ref{mod} we are grateful to Jakob Stix for providing  the following sharp bound.

A \textbf{Hurwitz curve} is a smooth projective curve $X$ which admits a branched Galois cover $X \to {\bf P}^1$ with only three branch points and ramification index $2$, $3$ and $7$. These are precisely the curves for which the Hurwitz-bound $|{\Aut(X)}| \leq 84(g-1)$ is an equality.

\begin{lemma}(Stix)
\label{lem:Hurwitzboundramification}
Let $X$ be a smooth projective geometrically connected curve of genus $g \geq 2$ over an algebraically closed field $k=\bar k$ of characteristic $0$. The number of points in $X$ which are fixed by a nontrivial automorphism of $X$ is bounded above by $82(g-1)$
\[
|\{P \in X \ ; \ \exists \id \not= \sigma \in \Aut(X): \ \sigma(P) = P\}| \leq 82(g-1).
\]
The bound is sharp and attained if and only if $X$ is a Hurwitz curve.
\end{lemma}
\begin{proof}
Let $G = \Aut(X)$ be the automorphism group and let $e_P$ denote the ramification index for points above $P \in X/G$ in the cover $X \to Y=X/G$. The number of points that we want to estimate is
\[
T = |{G}| \cdot \sum_{P \in Y} \frac{1}{e_P}.
\]
Let $B = |\{P \in Y \ ; \ e_P > 1\}|$ be the number of branch points. The Riemann Hurwitz formula tells us
\begin{align*}
2g-2 & = |{G}| (2g_Y - 2) + \sum_{P \in Y} |G| (1 - \frac{1}{e_P}) = |{G}| (2g_Y - 2 + B) - T \\
& = |{G}| (2g_Y - 2 + B - \sum_{P \in Y} \frac{1}{e_P}).
\end{align*}
If $g_Y \geq 1$, then since $1- \frac{1}{e_p} \geq \frac{1}{2} \geq \frac{1}{e_P}$ we are done because of 
\[
T \leq \sum_{P \in Y} |G| (1 - \frac{1}{e_P}) = 2 g-2 - |G| (2g_Y - 2) \leq 2g-2.
\]

So from now on we assume $g_Y = 0$.  Since $2g-2 > 0$, we must have that 
\[
B - 2 > \sum_{P \in Y} \frac{1}{e_P}.
\]
If $B \geq 5$, then 
\[
B - 2 - \sum_{P \in Y} \frac{1}{e_P} \geq B \cdot \frac{1}{2} - 2 \geq \frac{1}{2}
\]
and so 
\[
|G| = \frac{2g-2}{B - 2 - \sum_{P \in Y} \frac{1}{e_P}} \leq 4(g-1).
\]
It follows that 

\[
T \leq \sum_{P \in Y}|G| (1 - \frac{1}{e_P}) = 2 g-2 + 2|G|  \leq 10(g-1).
\]

If $B = 4$, then 
\[
B - 2 - \sum_{P \in Y} \frac{1}{e_P} \geq  2  - \frac{1}{2}- \frac{1}{2}- \frac{1}{2}- \frac{1}{3} = \frac{1}{6},
\]
hence 
\[
|G| \leq 12(g-1) \quad \text{ and } \quad T \leq 26(g-1).
\]
It remains to discuss the case of $B = 3$. Here, as in the proof of the Hurwitz bound, the minimal positive value of 
\[
B - 2 - \sum_{P \in Y} \frac{1}{e_P}
\]
is attained for ramification indices $2$, $3$ and $7$ leading to the Hurwitz bound $|G| \leq 84(g-1)$. But now 
\[
T = |G| \cdot (2g_Y - 2 + B) - 2(g-1) = |G| - 2(g-1) \leq 82(g-1). \qedhere
\]
\end{proof}

\end{document}